\documentclass[11pt]{article}
\usepackage{graphicx}
\usepackage{mathrsfs}

\usepackage{amsfonts}
\usepackage{amsmath}
\usepackage{amssymb}
\newtheorem{thm}{Theorem}[section]
\newtheorem{lem}[thm]{Lemma}

\newcommand{\ba}{\begin{array}}
\newcommand{\ea}{\end{array}}
\newcommand{\bt}{\begin{tabular}}
\newcommand{\et}{\end{tabular}}
\newcommand{\btb}{\begin{table}}
\newcommand{\etb}{\end{table}}
\newcommand{\bc}{\begin{center}}
\newcommand{\ec}{\end{center}}
\newcommand{\bea}{\begin{eqnarray}}
\newcommand{\eea}{\end{eqnarray}}
\newcommand{\Bea}{\begin{eqnarray*}}
\newcommand{\Eea}{\end{eqnarray*}}
\newcommand{\beq}{\begin{equation}}
\newcommand{\eeq}{\end{equation}}

\begin{document}
\baselineskip 16.5pt

\title{Divisibility by 2 of Stirling numbers of the second kind and their differences}

\author{ Jianrong Zhao\\
{\it School of Economic Mathematics, Southwestern
University of }\\ {\it Finance and Economics, Chengdu 610074, P.R. China}\\
Email: mathzjr@foxmail.com\\
\\
Shaofang Hong\thanks{Corresponding Author. The work was supported
partially by National Science Foundation of China Grant \#11371260
and by the Ph.D. Programs Foundation of Ministry of Education
of China Grant \#20100181110073.}\\
{\it Yangtze Center of Mathematics, Sichuan University, Chengdu 610064, P.R. China}\\
Email: sfhong@scu.edu.cn, s-f.hong@tom.com, hongsf02@yahoo.com\\
\\
Wei Zhao\\
{\it Mathematical College, Sichuan University, Chengdu 610064, P.R. China}}
\date{}
\maketitle


\noindent{\bf Abstract.} Let $n,k,a$ and $c$ be positive integers
and $b$ be a nonnegative integer. Let $\nu_2(k)$ and $s_2(k)$ be the
2-adic valuation of  $k$ and  the sum of binary digits of $k$,
respectively. Let $S(n,k)$ be the Stirling number of the second
kind. It is shown that $\nu_2(S(c2^n,b2^{n+1}+a))\geq s_2(a)-1,$
where $0<a<2^{n+1}$ and $2\nmid c$. Furthermore, one gets that
$\nu_2(S(c2^{n},(c-1)2^{n}+a))=s_2(a)-1$, where $n\geq 2$, $1\leq
a\leq 2^n$ and  $2\nmid c$. Finally, it is proved that if $3\leq k\leq
2^n$ and  $k$ is not a power of 2 minus 1, then
$\nu_2(S(a2^{n},k)-S(b2^{n},k))=n+\nu_2(a-b)-\lceil\log_2k\rceil
+s_2(k)+\delta(k), $ where $\delta(4)=2$, $\delta(k)=1$ if $k>4$ is
a power of 2, and $\delta(k)=0$ otherwise. This confirms a
conjecture of Lengyel  raised in 2009 except when $k$ is a power of
2 minus 1.

\vspace{1mm}

\noindent \textbf{Keywords:}  Stirling numbers of the second kind,
Congruences for power series, Bell polynomial, 2-Adic valuation

\vspace{1mm}

\noindent {\small \textbf{MR(2000) Subject Classification:} Primary
11B73, 11A07}


\section{Introduction and the statements of main results}

The Stirling numbers of the second kind $S(n,k)$ is defined for $n\in\mathbb{N}$
and positive integer $k\leq n$ as the number of ways to partition a set of  $n$
elements into exactly $k$ non-empty subsets. It satisfies the recurrence relation
\begin{align*}
  S(n,k)=S(n-1,k-1)+kS(n-1,k),
\end{align*}
with initial condition $S(0,0)=1$ and $S(n,0)=0$ for $n>0$.
 There is also an explicit formula in terms of binomial coefficients given by
\begin{align}\label{1.1}
  S(n,k)=\frac{1}{k!}\sum_{i=0}^k(-1)^i{k\choose i}(k-i)^n.
\end{align}
Divisibility properties of Stirling numbers have been studied from a
number of different perspectives. It is known that for each fixed
$k$, the sequence $\{S(n,k),n\geq k\}$ is periodic modulo prime
powers. The length of this period has been studied by Carlitz
\cite{[Ca]} and Kwong \cite{[Kw]}. Chan and Manna \cite{[CM]}
characterized $S(n,k)$ modulo prime powers in terms of binomial
coefficients. In fact, they gave explicit formulas for $S(n,k)$
modulo 4, then for $S(n,a2^m)$ modulo $2^m$, where $m\ge3$, $a>0$ and
$n\ge a2^m+1$, and finally for $S(n,ap^m)$ modulo $p^m$ with $p$
being an odd prime.

Divisibility properties of integer sequences are often expressed in
terms of $p$-adic valuations. Given a prime $p$ and a positive
integer $m$, there exist unique integers $a$ and $n$, with $p\nmid
a$ and $n\geq 0$, such that $m=ap^n$. The number $n$ is called the
{\it $p$-adic valuation} of $m$, denoted by $n=\nu_p(m)$. The numbers
$\min\{\nu_p(k!S(n,k)):m\leq k\leq n\}$ are important in algebraic
topology, see, for example, \cite{[BD], [CK], [D2], [D3], [D4],
[Lu1], [Lu2]}. Some work on evaluating $\nu_p(k!S(n,k))$ has appeared
in above papers as well as in \cite{[Cl],[D1],[Y]}. Amdeberhan,
Manna and Moll \cite{[AMM]} investigated the 2-adic valuations of
Stirling numbers of the second kind and computed $\nu_2(S(n,k))$ for
$k\le 4$. They also raised an interesting conjecture on the
congruence classes of $S(n, k)$, modulo powers of 2. Recently,
Bennett and Mosteig \cite{[BM]} used computational methods to justify
this conjecture if $k\le 20$. But this conjecture is still kept open
if $k\ge21$.

This paper deals with the 2-adic valuations of the Stirling numbers
of the second kind. Lengyel  \cite{[Le1]} studied the 2-adic
valuations of $S(n,k)$ and conjectured, proved by Wannemacker
\cite{[W]}, $\nu_2(S(2^n,k))=s_2(k)-1,$ where $s_2(k)$ means the
base $2$ digital sum of $k$.  Using Wannemacker's result, Hong, Zhao
and Zhao \cite{[HZZ]} proved that $\nu_2(S(2^n+1,k+1))=s_2(k)-1$,
which confirmed another conjecture of Amdeberhan,
Manna and Moll \cite{[AMM]}.  Lengyel \cite{[Le2]} showed that if
$1\leq k\leq 2^n$, then $\nu_2(S(c2^n,k))=s_2(k)-1$ for any positive
integer $c$. Meanwhile, Lengyel \cite{[Le2]} proved that
$\nu_2(S(c2^n,k))\geq s_2(k)-1$ if $c\geq1$ is an odd integer and
$1\le k\leq 2^{n+1}$. Actually, a more general result is true.
That is, one has
\begin{thm}\label{thm 1}
Let $n,a,b, c\in\mathbb{N}$ with $0<a<2^{n+1}$, $b2^{n+1}+a\leq c2^n$
and $c\geq 1$ being odd. Then
$$
\nu_2(S(c2^n,b2^{n+1}+a))\geq s_2(a)-1.
$$
\end{thm}

If one picks $b=\frac{c-1}{2}$ and $1\leq a\leq 2^n$, then the lower bound
in Theorem \ref{thm 1} is arrived as the following result shows.

\begin{thm}\label{thm 2}
Let $a,c,n\in\mathbb{N}$ with $c\geq 1$ being odd, $n\geq 2$ and
$1\leq a\leq 2^n$. Then
$$
\nu_2(S(c2^{n},(c-1)2^{n}+a))=s_2(a)-1.
$$
\end{thm}

Another interesting property is related to the difference of Stirling
numbers of the second kind. Lengyel \cite{[Le2]} studied the 2-adic
valuations of the difference  $S(c2^{n+1},k)-S(c2^n,k)$ with
$1\leq k\leq 2^n$ and $c\geq 1$ odd. In the meantime, Lengyel
posed the following conjecture.\\

\noindent{\bf Conjecture 1.1.} {\it\cite{[Le2]} Let $n, k, a, b\in
\mathbb{N}$, $c\geq 1$ being odd and $3\leq k\leq 2^n$. Then
\begin{align}
  \nu_2(S(c2^{n+1},k)-S(c2^{n},k))=n+1-f(k)\label{ad1}
\end{align}
and
$$\nu_2(S(a2^{n},k)-S(b2^{n},k))=n+1+\nu_2(a-b)-f(k)$$
for some function $f(k)$ which is independent of $n$.}\\

As usual, for any real number $x$, let $\lceil x\rceil$ and
$\lfloor x\rfloor$ denote the smallest integer no less than $x$ and
the biggest integer no more than $x$, respectively. Note that
Lengyel \cite{[Le2]} proved that (\ref{ad1}) is true for any integer
$k$ with $s_2(k)\leq 2$. Lengyel \cite{[Le2]} also noticed that for
small values of $k$, numerical experimentation suggests that
$f(k)=1+\lceil\log_2k\rceil-s_2(k)-\delta(k)$, where $\delta(4)=2$
and otherwise it is zero except if $k$ is a power of two or one
less, in which cases $\delta(k)=1$. The present paper focuses on
investigating Conjecture 1.1. One has the following result.

\begin{thm}\label{thm 4}
Let $n,k, a, b\in \mathbb{N}$, $c\geq 1$ being odd, $3\leq k\leq
2^n$, and $a>b$. If $k$ is not a power of 2 minus 1, then
\begin{align}\label{c2}
\nu_2(S(a2^{n},k)-S(b2^{n},k))=n+\nu_2(a-b)-\lceil\log_2k\rceil
+s_2(k)+\delta(k),
\end{align}
where $\delta(4)=2$, $\delta(k)=1$ if $k>4$ is a power of 2,
and $\delta(k)=0$ otherwise. In particular,
\begin{align}\label{c1}
\nu_2(S(c2^{n+1},k)-S(c2^{n},k))=n-\lceil\log_2k\rceil
+s_2(k)+\delta(k).
\end{align}
\end{thm}
By Theorem \ref{thm 4}, one knows that Conjecture 1.1 holds except
when $k$ is a power of 2 minus 1.

In order to prove Theorem \ref{thm 4}, one needs a special case of
the 2-adic valuation of $S(n,k)$, which can be stated as follows.
\begin{thm}\label{thm 3}
Let $a,b,c,m,n\in \mathbb{Z}^+$, $1\leq a<2^{n+1}$,
$m\geq n+2+\lfloor\log_2b\rfloor$ and $c\geq1$ being odd. Then
  $$
  \nu_2(S(c2^m+b2^{n+1}+2^n,b2^{n+2}+a))
  \left\{
  \begin{array}{lc}
   =n, & \text{if} ~a=2^{n+1}-1,\\
\geq s_2(a), & \text{if}~ a<2^{n+1}-1.
  \end{array}
  \right.
  $$
\end{thm}

This paper is organized as follows. Some preliminary results are presented
in Section 2. Then the proofs of Theorems 1.1 and 1.2 are given in Section 3.
Consequently, Section 4 is devoted to the proof of Theorem 1.4. Finally,
in Section 5, one uses Theorems 1.1 and 1.4 to show Theorem 1.3.

\section{Lemmas}

Several well-known results, which are needed for the proofs of
the main results, are given in this section.

\begin{lem}\label{lem -1} \cite{[R]} (Legendre)
Let $n\in\mathbb{N}$. Then $\nu_2\big(n!)=n-s_2(k).$
\end{lem}

\begin{lem}\label{lem -2}\cite{[Ku]} (Kummer)
Let $k$ and $n\in\mathbb{N}$ be such that $k\le n$. Then
$\nu_2({n\choose k})=s_2(k)+s_2(n-k)-s_2(n).$ Moreover,
$s_2(k)+s_2(n-k)\ge s_2(n).$
\end{lem}

\begin{lem}\cite{[Le2]}\label{lem 1} 
Let $k,n,c\in\mathbb{N}$ and $1\leq k\leq 2^{n}$. Then
$\nu_2(S(c2^n,k))=s_2(k)-1.$
\end{lem}

\begin{lem}\cite{[Le2]}\label{lem 2} 
 Let $k,n,c\in\mathbb{N}$, $2^n< k<2^{n+1}-1$ and $c\geq 3$ be an odd integer. Then
$\nu_2(S(c2^n,k))\geq s_2(k)$ and $\nu_2(S(c2^n,2^{n+1}-1))=n.$
\end{lem}
\begin{lem}\cite{[Le2]}\label{lem -3} 
 Let $m,n,c\in\mathbb{N}$ and  $0\le m<n$.  Then
$\nu_2(S(c2^n+2^m,2^{n}))=n-1-m.$
\end{lem}
\begin{lem}\cite{[W]}\label{lem 0} 
Let $k,n,m\in\mathbb{N}$ and $0\leq k\leq n+m$. Then
\begin{align*}
 S(n+m,k)=\sum_{j=1}^k\sum_{i=0}^j{j\choose i}\frac{(k-i)!}{(k-j)!}S(n,k-i)S(m,j).
\end{align*}
\end{lem}
\begin{lem}\cite{[AD]}\label{lem 3}
For $r\geq \max(k_1,k_2)+2,$ one has
$$\frac{k_1!k_2!(r-1)!}{(k_1+k_2+1)!}S(k_1+k_2+2,r)=
\sum_{i=1}^{r-1}(i-1)!(r-i-1)!S(k_1+1,i)S(k_2+1,r-i).$$
\end{lem}

\begin{lem} \cite{[J]}\label{lem 10}
Let $m,n,v\in\mathbb{N}$, $v\geq 1$ and $p$ be a prime number. Then
\begin{align}\label{1}
B_{m+np^v}(x)\equiv\sum_{j=0}^n{n\choose
j}(x^p+x^{p^2}+\cdots+x^{p^v})^{n-j}B_{m+j}(x)\mod{\frac{np}{2}\mathbb{Z}_p[x]},
\end{align}
where the Bell polynomials are defined by
\begin{align}\label{2}
B_n(x)=\sum_{k=0}^nS(n,k)x^k,n\geq 0.
\end{align}
\end{lem}

Let $n=\sum_{\lambda=0}^{\infty}\varepsilon_\lambda(n)2^\lambda$ with
$\varepsilon_\lambda(n)\in\{0,1\}$. Then $s_2(n)=\sum_{\lambda=0}
^{\infty}\varepsilon_\lambda(n)$. Further, one has the following result.

\begin{lem} \label{lem 4}
Let $m$ and $n\in\mathbb{N}$. Then $s_2(m+n)=s_2(m)+s_2(n)$ if and only if
$\varepsilon_\lambda(m)+\varepsilon_\lambda(n)=\varepsilon_\lambda(m+n)$
for all $\lambda\in\mathbb{N}$.
\end{lem}

\begin{proof}
This lemma follows immediately from the proof of Lemma 1 in \cite{[W]}. \hfill$\Box$
\end{proof}

\begin{lem}\label{lem 5}
Let $n, a\in \mathbb{N}$ and $1\leq a<2^{n+1}$. Define the set $J$ of positive integers
by $J:=\{1\le j\le 2^n\mid s_2(2^{n+1}+a-j)+s_2(j)=s_2(2^{n+1}+a)\}$.
Then $|J|=2^{s_2(a)}-1$ if $1\leq a\leq 2^n$, and $|J|=2^{s_2(a)-1}$ if $2^n<a<2^{n+1}$.
\end{lem}

\begin{proof} For any positive integer $d$, define $M_d:=\{\lambda\in \mathbb{N}
\mid \varepsilon_\lambda(d)=1\}$.
Then $d=\sum_{\lambda\in M_d}2^\lambda$ and $s_2(d)=|M_d|$. By Lemma
\ref{lem 4} one knows that $s_2(2^{n+1}+a-j)+s_2(j)=s_2(2^{n+1}+a)$ if and only if
\begin{align}
\varepsilon_\lambda(j)+\varepsilon_\lambda(2^{n+1}+a-j)
=\varepsilon_\lambda(2^{n+1}+a)
 \label{eq16}
\end{align}
for all $\lambda\in \mathbb{N}$. Therefore for any given $\lambda\in \mathbb{N}$,
$\varepsilon_\lambda(j)=0$ or 1 if $\varepsilon_\lambda(2^{n+1}+a)=1$, and
$\varepsilon_\lambda(j)=0$ if $\varepsilon_\lambda(2^{n+1}+a)=0$.
It then follows that for any given integer
 $1\le a \le 2^n$,  $j\in J$ if and only if $M_j\subseteq M_{a}$ and  $M_j\neq\varnothing$.
So $|J|=2^{|M_a|}-1=2^{s_2(a)}-1$ if $1\leq a\leq 2^n$.

Now let $2^n< a<2^{n+1}$.  So if $j=2^n$, then one can check that
$s_2(2^{n+1}+a-2^n)+s_2(2^n)=s_2(2^{n+1}+a)$. This implies that  $2^n\in J$.
On the other hand, since $1<a-2^n<2^n$, one gets that $j\in J\setminus\{2^n\}$
if and only if $M_j\subseteq M_{a-2^n}$ and $M_j\neq\varnothing$.
Hence $|J|=2^{|M_{a-2^n}|}=2^{s_2(a)-1}$ if $2^n< a<2^{n+1}$.
The proof of Lemma \ref{lem 5} is complete.   \hfill$\Box$
\end{proof}

\begin{lem}\label{lem 7}
Let $n, a, c\in \mathbb{N}$ with $c\geq1$ being odd and $1\leq a\leq 2^{n}$. Then
\begin{align}\label{add11}
s_2(c2^n-a)=s_2(c)+n-\nu_2(a)-s_2(a).
\end{align}
\end{lem}

\begin{proof} If $a=2^n$, then it is easy to check that (\ref{add11})
is true. Now let $1\leq a<2^n$. One can write
$a=\sum_{\lambda=\nu_2(a)}^{n-1} \varepsilon_\lambda(a)2^\lambda$.
Clearly $s_2(a)=\sum_{\lambda=\nu_2(a)}^{n-1}
\varepsilon_\lambda(a)$ and $\varepsilon_{\nu_2(a)}(a)=1$. Then
\begin{align}\label{add12}
c2^n-a&=(c-1)2^n+2^n-a\nonumber\\
&=(c-1)2^n+\Big(2^{\nu_2(a)}+\sum_{\lambda=\nu_2(a)}^{n-1}2^\lambda\Big)-
\sum_{\lambda=\nu_2(a)}^{n-1}\varepsilon_\lambda(a)2^\lambda\nonumber\\
&=(c-1)2^n+\sum_{\lambda=\nu_2(a)}^{n-1}(1-\varepsilon_\lambda(a))2^\lambda+2^{\nu_2(a)}.
\end{align}

Since $s_2(c-1)=s_2(c)-1$, by (\ref{add12}) one has
\begin{align*}
 s_2(c2^n-a)&=s_2(c-1)+\sum_{\lambda=\nu_2(a)}^{n-1}(1-\varepsilon_\lambda(a))+1\\
 &=s_2(c)+\sum_{\lambda=\nu_2(a)}^{n-1}(1-\varepsilon_\lambda(a))\\
 &=s_2(c)+n-\nu_2(a)-s_2(a)
\end{align*}
as required. This completes the proof of Lemma \ref{lem 7}.
\hfill$\Box$\\
\end{proof}

\begin{lem}\label{lem 9} \cite{[HZZ]}
Let $N\geq 2$ be an integer and $r$, $t$ be odd numbers. For any
$m\in \mathbb{Z}^+,$  one has $\nu_2((r2^N-1)^{t2^m}-1)=m+N$.
\end{lem}

\section{Proofs of Theorems 1.1 and 1.2}

In this section, one uses induction and Lemmas 2.1 to 2.4 and 2.6 as
well as 2.7 to show Theorems 1.1 and 1.2. One begins with the proof
of Theorem 1.1.\\

\noindent{\it Proof of Theorem 1.1.} If $b=0$, then Theorem \ref{thm 1}
is true by Lemmas \ref{lem 1} and \ref{lem 2}. In what follows one lets
$b\geq 1$. There exists an unique integer $e\geq 0$ such that
$2^e\leq b<2^{e+1}$. One shows Theorem \ref{thm 1} using induction on $e$.
First one treats the case $e=0$, i.e., $b=1$. Using Lemma
\ref{lem 0} with $n$, $m$ and $k$ replaced by $(c-1)2^n$,
$2^{n}$ and $2^{n+1}+a$, respectively, one has
\begin{align}\label{(3.1)}
  S(c2^n,2^{n+1}+a) &=\sum_{j=1}^{2^{n+1}+a}\sum_{i=0}^j
  f(i,j)=\sum_{j=1}^{2^{n}}\sum_{i=0}^{2^n}f(i,j),
\end{align}
where
$$
f(i,j):={j\choose i}\frac{(2^{n+1}+a-i)!}{(2^{n+1}+a-j)!}
S((c-1)2^n,2^{n+1}+a-i)S(2^{n},j).
$$

Since $c$ is an odd integer,  $\nu_2((c-1)2^n)\geq n+1$. It then
follows from Lemmas \ref{lem -1}, \ref{lem 1} and \ref{lem 2} that
\begin{align}
\nu_2(f(i,j))&\geq
\nu_2\bigg(\frac{(2^{n+1}+a-i)!}{(2^{n+1}+a-j)!}\bigg)
+\nu_2(S((c-1)2^n,2^{n+1}+a-i)) +\nu_2(S(2^{n},j))\nonumber\\
&\geq \nu_2((2^{n+1}+a-i)!)-\nu_2((2^{n+1}+a-j)!)+s_2(2^{n+1}+a-i)-1+s_2(j)-1\nonumber\\
&= (j-i)+s_2(2^{n+1}+a-j)-s_2(2^{n+1}+a-i)+s_2(2^{n+1}+a-i)+s_2(j)-2\nonumber\\
&\geq s_2(2^{n+1}+a-j)+s_2(j)-2\label{eq 1}
\end{align}
since $j\geq i$. By Lemma \ref{lem -2} one knows that
$$s_2(j)+s_2(2^{n+1}+a-j)\ge s_2(2^{n+1}+a).$$
So by (\ref{eq 1}) and noting that $0<a<2^{n+1}$, one obtains
\begin{align}\label{(3.2)}
\nu_2(f(i,j))\geq s_2(2^{n+1}+a)-2=s_2(a)-1.
\end{align}
It then follows from (\ref{(3.1)}) and (\ref{(3.2)}) that
$$
\nu_2( S(c2^n,2^{n+1}+a))\geq \min_{0\leq i\leq j\leq
2^n}\{\nu_2(f(i,j))\} \geq s_2(a)-1.
$$
Hence Theorem \ref{thm 1} is true if $e=0$. In what follows, one lets $e\geq1$.

Assume that Theorem 1.1 is true for the case $t$ with $t\le e-1$.
Then $\nu_2(S(c2^n, b2^{n+1}+a))\geq s_2(a)-1$ for any integers $b$
with $0\leq b<2^e$. In the following one proves that Theorem \ref{thm
1} is true for the case $e$. This is equivalent to showing that
Theorem \ref{thm 1} is true for all integers $b\in[2^e,2^{e+1})$,
which will be done in what follows.

Let $b\in[2^e,2^{e+1})$ be any given integer. Since $c2^n\geq
b2^{n+1}+a$, there exist two positive integers $c_1$ and $c_2$ such
that $c=c_1+c_22^{\nu_2(b)+1}$ and $c_1<2^{\nu_2(b)+1}$. So by Lemma
\ref{lem 0}
\begin{align}\label{101}
S(c2^n,b2^{n+1}+a)&=\sum_{j=1}^{c_12^{n}}\sum_{i=0}^{j}g(i,j),
\end{align}
where
$$
g(i,j):={j\choose i}\frac{(b2^{n+1}+a-i)!}{(b2^{n+1}+a-j)!}
  S(c_22^{n+\nu_2(b)+1},b2^{n+1}+a-i)S(c_12^{n},j).
$$

\noindent{\bf Claim 1.} One has
\begin{align}
 \nu_2(S(c_22^{n+\nu_2(b)+1},b2^{n+1}+a-i)) \geq
s_2(b2^{n+1}+a-i)-s_2(b).\label{102}
\end{align}

Let's now prove Claim 1. If $\nu_2(c_2)+\nu_2(b)\geq e$,
then $b2^{n+1}+a-i<2^{e+n+2} \le
2^{\nu_2(b)+\nu_2(c_2)+n+2}$ since $a<2^{n+1}$ and $2^e\leq
b<2^{e+1}$. So by Lemmas \ref{lem 1} and \ref{lem 2}, one obtains that
\begin{align*}
  & \nu_2(S(c_22^{n+\nu_2(b)+1},b2^{n+1}+a-i))\nonumber\\
  &=\nu_2\bigg(S\Big(\frac{c_2}{2^{\nu_2(c_2)}}2^{n+\nu_2(b)+\nu_2(c_2)+1},
  b2^{n+1}+a-i\Big)\bigg)\nonumber\\
  &\geq s_2(b2^{n+1}+a-i)-1\nonumber\\
  &\geq s_2(b2^{n+1}+a-i)-s_2(b)
\end{align*}
as desired. So Claim 1 is proved in this case.

If $\nu_2(c_2)+\nu_2(b)\leq e-1$, then one can write
$b=b_12^{\nu_2(c_2)+\nu_2(b)+1}+b_2 $ for some integers
$0<b_1<2^{e-\nu_2(c_2)-\nu_2(b)}$ and $2^{\nu_2(b)}\leq
b_2<2^{\nu_2(c_2)+\nu_2(b)+1}$ since $2^e\leq b<2^{e+1}$. One can
deduce that $s_2(b2^{n+1}+a-i)=s_2(b_22^{n+1}+a-i)+s_2(b_1)$. It
then follows from the inductive hypothesis that
\begin{align*}
  &\nu_2(S(c_22^{n+\nu_2(b)+1},b2^{n+1}+a-i))\nonumber\\
  &=\nu_2\bigg(S\bigg(\frac{c_2}{2^{\nu_2(c_2)}}2^{n+\nu_2(b)+\nu_2(c_2)+1},
  b_12^{n+\nu_2(b)+\nu_2(c_2)+2}
  +b_22^{n+1}+a-i\bigg)\bigg)\nonumber\\
  &\geq s_2(b_22^{n+1}+a-i)-1\nonumber\\
  &=s_2(b2^{n+1}+a-i)-s_2(b_1)-1\nonumber\\
  &\geq s_2(b2^{n+1}+a-i)-s_2(b)
\end{align*}
as required. So Claim 1 is true for this case. This concludes
the proof of Claim 1.

\noindent {\bf Claim 2.} For all the integers $i$ and $j$ such that
$0\leq i\leq j\leq c_12^{n}$ with $c_1<2^{\nu_2(b)+1}$, one has
\begin{align}
\nu_2(g(i, j))\ge s_2(a)-1.\label{105}
\end{align}
Suppose that Claim 2 is true. Then from  (\ref{101}) and Claim 2,
one deduces that
$$
\nu_2( S(c2^n,b2^{n+1}+a))\geq \min_{0\leq i\leq j\leq
c_12^n}\{\nu_2(g(i,j))\} \geq s_2(a)-1.
$$
In other words, Theorem \ref{thm 1} holds if $b\in [2^e,2^{e+1})$.
To finish the proof of Theorem \ref{thm 1}, it remains to show that
Claim 2 is true which will be done in the following.

If $1\leq j<2^{n+1}$, then by Lemmas \ref{lem 1} and \ref{lem 2} one
has $\nu_2(S(c_12^n,j))\geq s_2(j)-1$. Thus using Lemmas \ref{lem
-1}-\ref{lem -2} and the Claim 1, one derives from
$a<2^{n+1}$ that
\begin{align*}
\nu_2(g(i,j))&\geq
\nu_2\bigg(\frac{(b2^{n+1}+a-i)!}{(b2^{n+1}+a-j)!}\bigg)+s_2(b2^{n+1}+a-i)-s_2(b)
+s_2(j)-1\nonumber\\
&\geq s_2(b2^{n+1}+a-j)-s_2(b2^{n+1}+a-i)+s_2(b2^{n+1}+a-i)-s_2(b)+s_2(j)-1\nonumber\\
&\geq s_2(b2^{n+1}+a-j)+s_2(j)-s_2(b)-1\nonumber\\
&\geq s_2(b2^{n+1}+a)-s_2(b)-1\nonumber\\
&=s_2(a)-1
\end{align*}
as required. Hence Claim 2 is true in this case.

If $2^{n+1}\leq j\leq c_12^n,$ then one may let $j=j_{1}2^{n+1}+j_2$
for some integers  $0\leq j_2<2^{n+1}$ and $j_{1}<2^{\nu_2(b)}$
since $c_1<2^{\nu_2(b)+1}$.  If $j_2=0$, i.e., $j=j_{1}2^{n+1}$,
then by (\ref{102}) and Lemmas \ref{lem -1}-\ref{lem -2},
noting that  $a<2^{n+1}$, one yields
\begin{align*}
\nu_2(g(i,j)) &\geq
\nu_2\Big(\frac{(b2^{n+1}+a-i)!}{(b2^{n+1}+a-j)!}\Big)+\nu_2(
S(c_22^{n+\nu_2(b)+1},b2^{n+1}+a-i)) \nonumber\\
&\geq s_2(b2^{n+1}+a-j)-s_2(b2^{n+1}+a-i)+s_2(b2^{n+1}+a-i)-s_2(b)\nonumber\\
&= s_2((b-j_{1})2^{n+1}+a)-s_2(b)\nonumber\\
&= s_2(b-j_{1})+s_2(a)-s_2(b)\nonumber\\
&\geq s_2(a)
\end{align*}
since $j_{1}<2^{\nu_2(b)}$ implying that $s_2(b-j_1)\ge s_2(b)$.
Hence Claim 2 is true if $j_2=0$. Now let $j_2\geq1$.
Since $j_{1}<2^{\nu_2(b)}\le 2^e$, by the inductive hypothesis one has
\begin{align}
\nu_2(S(c_12^n,j))=\nu_2(S(c_12^n,j_{1}2^{n+1}+j_2))\geq
s_2(j_2)-1.\label{107}
\end{align}
Thus by Lemmas  \ref{lem -1}-\ref{lem -2}, (\ref{102}) and (\ref{107}) one obtains
\begin{align*}
\nu_2(g(i,j))&\geq
\nu_2\Big(\frac{(b2^{n+1}+a-i)!}{(b2^{n+1}+a-j)!}\Big)+\nu_2(
S(c_22^{n+\nu_2(b)+1},b2^{n+1}+a-i))+S(c_12^n,j) \nonumber\\
&\geq s_2(b2^{n+1}+a-j)-s_2(b2^{n+1}+a-i)+s_2(b2^{n+1}+a-i)-s_2(b)+s_2(j_2)-1\nonumber\\
&= s_2(b2^{n+1}+a-j)+s_2(j_2)-s_2(b)-1\nonumber\\
&= s_2((b-j_{1})2^{n+1}+a-j_2)+s_2(j_2)-s_2(b)-1\nonumber\\
&\geq  s_2((b-j_{1})2^{n+1}+a)-s_2(b)-1\nonumber\\
&=s_2(b-j_{1})+s_2(a)-s_2(b)-1\nonumber\\
&\geq s_2(a)-1
\end{align*}
since $s_2(b-j_1)\ge s_2(b)$. Hence Claim 2 holds if $j_2\ge
1$. So Claim 2 is proved.

This completes the proof of Theorem \ref{thm 1}. \hfill$\Box$\\

Consequently, one turns attention to the proof of Theorem 1.2.\\

\noindent{\it Proof of Theorem 1.2.}
If $a=2^n$, then by definition of Stirling numbers
of the second kind, one has
$$S(c2^{n},(c-1)2^{n}+a)=S(c2^{n},c2^n)=1.$$
This implies that $\nu_2(S(c2^{n},c2^n))=s_2(2^n)-1.$ So Theorem
\ref{thm 2} is true in this case.

Now let $1\leq a<2^n$ and $b=\frac{c-1}{2}$. Then
$$S(c2^{n},(c-1)2^{n}+a)=S(b2^{n+1}+2^{n},b2^{n+1}+a).$$
To prove Theorem \ref{thm 2}, it is sufficient to show that
\begin{align}
 \nu_2(S(b2^{n+1}+2^{n},b2^{n+1}+a))=s_2(a)-1.  \label{h1}
\end{align}

For $t\in \mathbb{N}$, define
\begin{align}\label{add10}
  A_t:=\{b\in \mathbb{N}\mid s_2(b)=t\}.
\end{align}
Then $\mathbb{N}=\bigcup_{t=0}^\infty A_t$. The proof is proceeded with
induction on $t$. First one considers the case $t=0$. If $b\in A_0$,
then $b=0$. By Lemma \ref{lem 1} one has
$$
\nu_2(S(b2^{n+1}+2^{n},b2^{n+1}+a))=\nu_2(S(2^{n},a))=s_2(a)-1.
$$
So Theorem \ref{thm 2} holds if $t=0$.

In the following let $t\geq1$. Assume that Theorem \ref{thm 2} is
true for the case $r$ with $r\leq t-1$. Then (\ref{h1}) holds
for any positive integers $b\in A_0\cup A_1\cup\cdots\cup
A_{t-1}.$ One will prove that Theorem \ref{thm 2} is true
for the case $t$, which is equivalent to showing (\ref{h1})
for all the integers $b\in A_t$.

Let $b\in A_t$ be a given integer. One first notices that
$$b2^{n+1}+a\geq\max(b2^{n+1}-1,2^{n}-1)+2.$$
Letting $k_1=b2^{n+1}-1, k_2=2^{n}-1$ and $r=b2^{n+1}+a$ in
Lemma \ref{lem 3} gives us that
\begin{align*}
&\frac{(b2^{n+1}-1)!(2^{n}-1)!}
{(b2^{n+1}+2^{n}-1)!}(b2^{n+1}+a-1)!S(b2^{n+1}+2^{n},b2^{n+1}+a)\\
=&\sum_{i=1}^{b2^{n+1}+a-1}(i-1)!(b2^{n+1}+a-i-1)!S(2^{n},i)
S(b2^{n+1},b2^{n+1}+a-i)\\
=&\sum_{i=a}^{2^{n}}\frac{1}{i(b2^{n+1}+a-i)}i!S(2^{n},i)
(b2^{n+1}+a-i)!S(b2^{n+1},b2^{n+1}+a-i).
\end{align*}
It follows that
\begin{align}\label{eq5}
(b2^{n+1}+a)!S(b2^{n+1}+2^{n},b2^{n+1}+a)
=\frac{(b2^{n+1}+2^{n}-1)!}{(b2^{n+1}-1)!(2^{n}-1)!}
\sum_{i=a}^{2^{n}}l(i),
\end{align}
where
$$l(i):=\frac{b2^{n+1}+a}{i(b2^{n+1}+a-i)}i!S(2^{n},i)
(b2^{n+1}+a-i)!S(b2^{n+1},b2^{n+1}+a-i).
$$

Write $b=(2b_0+1)2^{\nu_2(b)}$ for some $b_0\in\mathbb{N}$. Clearly
$s_2(b_0)=s_2(b)-1=t-1$ since $b\in A_t$. Then $b_0\in A_{t-1}$. It
then follows from Lemma \ref{lem -1} that
\begin{align}
   \nu_2\bigg(\frac{(b2^{n+1}+2^{n}-1)!}{(b2^{n+1}-1)!(2^{n}-1)!}\bigg)
   =&\nu_2((b2^{n+1}+2^{n}-1)!)-\nu_2((b2^{n+1}-1)!)-\nu_2((2^{n}-1)!)\nonumber\\
   =& 1-s_2(b2^{n+1}+2^{n}-1)+s_2(b2^{n+1}-1)+s_2(2^{n}-1)\nonumber\\
   =& 1-s_2(b2^{n+1})+s_2(b_02^{n+\nu_2(b)+2}+2^{n+\nu_2(b)+1}-1)\nonumber\\
   =& 1-s_2(b)+s_2(b_0)+n+\nu_2(b)+1\nonumber\\
   =&n+\nu_2(b)+1.\label{eq6}
\end{align}
On the other hand, one has
 \begin{align}
 \nu_2((b2^{n+1}+a)!)&=(b2^{n+1}+a)-s_2((b2^{n+1}+a))\nonumber\\
 &=b2^{n+1}+a-s_2(b)-s_2(a).\label{eq7}
 \end{align}
So in order to show that (\ref{h1}) is true, by
(\ref{eq5})-(\ref{eq7}) one only needs to show that
 \begin{align}
\nu_2\bigg(\sum_{i=a}^{2^{n}}l(i)\bigg)=\big(b2^{n+1}+a\big)
-\big(s_2(b)+\nu_2(b)+n+2\big).\label{109}
 \end{align}
To do so, one discusses the 2-adic valuation of $l(i)$ with
$a\leq i\leq 2^{n}$ in what follows.

Since $b_0\in A_{t-1}$ and
$0<2^{n+\nu_2(b)+1}+a-i\leq2^{n+\nu_2(b)+1}$, by the inductive
hypothesis and Lemma \ref{lem 1}, one can derive that
\begin{align}
&\nu_2(S(b2^{n+1},b2^{n+1}+a-i))\nonumber\\
=&\nu_2\bigg(S\bigg(b_02^{n+\nu_2(b)+2}+2^{n+\nu_2(b)+1},
b_02^{n+\nu_2(b)+2}+2^{n+\nu_2(b)+1}+a-i\bigg)\bigg)\nonumber\\
=&s_2(2^{n+\nu_2(b)+1}+a-i)-1\nonumber\\
=&s_2((2b_0+1)2^{n+\nu_2(b)+1}+a-i)-s_2(b_0)-1\nonumber\\
=&s_2(b2^{n+1}+a-i)-s_2(b)\label{eq8}
\end{align}
since $b=(2b_0+1)2^{\nu_2(b)}$ and $s_2(b)=s_2(b_0)+1$. Furthermore,
by Lemmas \ref{lem -1}, \ref{lem 1} and (\ref{eq8}) one can compute that
\begin{align}
  &\nu_2\big(i!S(2^{n},i)(b2^{n+1}+a-i)!S(b2^{n+1},b2^{n+1}+a-i)\big)\nonumber\\
  =&i-s_2(i)+s_2(i)-1+(b2^{n+1}+a-i)-s_2(b2^{n+1}+a-i)+s_2(b2^{n+1}+a-i)-s_2(b)\nonumber \\
  =&(b2^{n+1}+a)-s_2(b)-1.\label{eq9}
\end{align}
Then by (\ref{eq9}) one has
\begin{align}
  \nu_2(l(i))&=(b2^{n+1}+a)-s_2(b)-1+\nu_2(b2^{n+1}+a)-\nu_2(i)-\nu_2(b2^{n+1}+a-i). \label{eq10}
\end{align}

If $i=a$, then  by (\ref{eq10}) and noticing that $a\le 2^n$, one gets that
\begin{align}
  \nu_2(l(a))&=(b2^{n+1}+a)-s_2(b)-1+\nu_2(b2^{n+1}+a)-\nu_2(a)-\nu_2(b2^{n+1})\nonumber\\
  &=\big(b2^{n+1}+a\big)-\big(s_2(b)+\nu_2(b)+n+2\big).\label{eq11}
\end{align}

If $a<i\leq 2^{n}$ and $\nu_2(i)\leq \nu_2(b2^{n+1}+a)$, then
\begin{align}
\nu_2(i)-\nu_2(b2^{n+1}+a)+\nu_2(b2^{n+1}+a-i)\leq
\nu_2(b2^{n+1}+a-i)<n. \label{add1}
\end{align}
It then follows from (\ref{eq10}) and (\ref{add1}) that
\begin{align}\label{eq12}
\nu_2(l(i))>(b2^{n+1}+a)-s_2(b)-1-n>\big(b2^{n+1}+a\big)-\big(s_2(b)+\nu_2(b)+n+2\big).
\end{align}

If $a<i\leq 2^{n}$ and $\nu_2(i)> \nu_2(b2^{n+1}+a)$, then one has
\begin{align}
\nu_2(i)-\nu_2(b2^{n+1}+a)+\nu_2(b2^{n+1}+a-i)=\nu_2(i)\leq n.
\label{add2}
\end{align}
So  by (\ref{eq10}) and (\ref{add2}) one has
\begin{align}\label{eq13}
  \nu_2(l(i))\geq(b2^{n+1}+a)-s_2(b)-1-n>\big(b2^{n+1}+a\big)-\big(s_2(b)+\nu_2(b)+n+2\big).
\end{align}
Thus the desired result (\ref{109}) follows immediately from (\ref{eq11}), (\ref{eq12})
and (\ref{eq13}). So (\ref{h1}) holds if $b\in A_t$, which implies that
Theorem \ref{thm 2} is true if $b\in A_t$.

The proof of Theorem \ref{thm 2} is complete. \hfill$\Box$

\section{Proof of Theorem 1.4}

The purpose of this section is to prove Theorem \ref{thm 3}. Note that
its proof is different from the proofs of Theorems \ref{thm 1} to
\ref{thm 2}. So one provides the details of the proof of Theorem
\ref{thm 3}. Throughout this section, one always lets
$a,b,c,m,n\in \mathbb{Z}^+$, $1\leq a<2^{n+1}$,
$m\geq n+2+\lfloor\log_2b\rfloor$ and $c\geq1$
being odd. For any integers $i$ and $j$ with $0\le i\le j\le
b2^{n+1}+2^n$, one defines
\begin{align}\label{201}
h(i,j):={j\choose
i}\frac{(b2^{n+2}+a-i)!}{(b2^{n+2}+a-j)!}S(c2^m,b2^{n+2}+a-i)S(b2^{n+1}+2^{n},j).
\end{align}
Let
\begin{align}\label{202}
\begin{array}{ll}
\Delta_1:=\displaystyle\sum_{j=1}^{2^n}\sum_{i=0}^{j}h(i,j),&
\Delta_2:=\displaystyle\sum_{j=2^n+1}^{2^{n+1}-2} \sum_{i=0}^{j}h(i,j),\\
\Delta_3:=\displaystyle\sum_{i=0}^{2^{n+1}-1}h(i,2^{n+1}-1),&
\Delta_4:=\displaystyle\sum_{j=b2^{n+1}+1}^{b2^{n+1}+2^n} \sum_{i=0}^jh(i,j).
\end{array}
\end{align}
First one uses the lemmas in Section 2 and Theorem \ref{thm 2}  to prove
the following result.

\begin{lem}\label{lem 6}
Each of the following is true:
\begin{enumerate}
\item[{\rm (i)}] For $l=1$ and $4$, one has
$ \nu_2(\Delta_l) \left\{
  \begin{array}{ll}=s_2(a)-1,&{\text if} ~1\leq a\leq 2^n~{\text and~}s_2(b)=1,\\
   \geq s_2(a),& {\text otherwise};
  \end{array}
  \right.
$

\item[{\rm (ii)}] $\nu_2(\Delta_2) \geq s_2(a)$;

\item[{\rm (iii)}]
$ \nu_2(\Delta_3)\left\{
  \begin{array}{ll}
   =n,& {\text if}~a=2^{n+1}-1~{\text and}~s_2(b)=1, \\
   \geq s_2(a),&{\text otherwise};
  \end{array}
  \right.
$

\item[{\rm (iv)}]
$ \nu_2(\Delta_1+\Delta_2+\Delta_3+\Delta_4)\left\{
  \begin{array}{ll}
   =n,& {\text if}~a=2^{n+1}-1~{\text and}~s_2(b)=1, \\
   \geq s_2(a),&{\text otherwise}.
  \end{array}
  \right.
$
\end{enumerate}
\end{lem}

\begin{proof}
Evidently, part (iv) follows immediately from parts (i)-(iii). So one needs only to
show parts (i)-(iii) which will be done in what follows. By Lemmas 2.1 and 2.2, one has
\begin{align}
&\nu_2\bigg({j\choose
i}\frac{(b2^{n+2}+a-i)!}{(b2^{n+2}+a-j)!}\bigg)\nonumber\\
=&s_2(i)+s_2(j-i)-s_2(j)+j-i+s_2(b2^{n+2}+a-j)-s_2(b2^{n+2}+a-i).\label{eq14}
\end{align}

(i). First one treats with $\Delta_1$. Let $1\leq j\leq 2^n$
and $0\leq i\leq j$. By Lemma \ref{lem 1}
\begin{align}
\nu_2(S(b2^{n+1}+2^{n},j))=s_2(j)-1.\label{add5}
\end{align}

Let $m>n+2+\lfloor\log_2b\rfloor$. Since $a<2^{n+1}$ and $1\leq
i\leq 2^n$, one has $b2^{n+2}+a-i<2^m$. By Lemma \ref{lem 1} one
obtains $\nu_2(S(c2^m, b2^{n+2}+a-i))=s_2(b2^{n+2}+a-i)-1$. Then
from (\ref{201}), (\ref{eq14}), (\ref{add5}) and Lemma
\ref{lem -2}, one obtains that
\begin{align}
\nu_2(h(i,j))& =s_2(i)+s_2(j-i)+j-i+s_2(b2^{n+2}+a-j)-2\nonumber\\
&\geq s_2(j)+s_2(b2^{n+2}+a-j)+j-i-2\nonumber\\
&\geq s_2((b2^{n+2}+a)-2\nonumber\\
&\geq s_2(a)-1, \label{eq15}
\end{align}
where equality holds if and only if $j=i$, $s_2(b)=1$ and
$s_2(b2^{n+2}+a-j)+s_2(j)=s_2(b2^{n+2}+a)$. So by (\ref{202}) one gets
\begin{align}
\Delta_1=2^{s_2(a)}\widetilde{\Delta}_1+2^{s_2(a)-1}\sum_{(i,j)\in
\widetilde{J}}\widetilde{h}(i,j),\label{add3}
\end{align}
where $\widetilde{\Delta}_1\in \mathbb{Z}^+$ and
$\widetilde{J}:=\{(i,j)\mid \widetilde{h}(i,j)~{\rm is}~{\rm odd},
1\le i\le j\le 2^n\}$. Then
\begin{align*}
\widetilde{J}&=\{(i,j)\mid  j=i, s_2(b)=1~ {\rm and}~ s_2(b2^{n+2}+a-j)+s_2(j)=s_2(b2^{n+2}+a)\}\\
&=\{(j,j)\mid   s_2(b)=1~ {\rm and}~ s_2(b2^{n+2}+a-j)+s_2(j)=s_2(b2^{n+2}+a)\}\\
&=\{1\le j\le 2^n\mid  s_2(b)=1~ {\rm and}~ s_2(2^{n+2}+a-j)+s_2(j)=s_2(2^{n+2}+a)\}.
\end{align*}
Thus by Lemma \ref{lem 5} one knows that
$|\widetilde{J}|=2^{s_2(a)}-1$ if $1\le a \le 2^n$ and
$2^{s_2(a)-1}$ else.

Furthermore,  by (\ref{add3}), one derives that $\nu_2(\Delta_1)$
equals $s_2(a)-1$ if $s_2(b)=1$ and $1\leq a\leq 2^n$, and is
greater than $s_2(a)$ otherwise.  So Lemma \ref{lem 6} (i) is true
if $l=1$ and $m>n+2+\lfloor\log_2b\rfloor$.

Now let $m=n+2+\lfloor\log_2b\rfloor$. If either
$2^n<a<2^{n+1}$, or $1\leq a\leq 2^n$ and $1\leq i<a$, then
one can check that the following is true:
$$2^m\le b2^{n+2}<b2^{n+2}+a-i<b2^{n+2}+a\le 2^{m+1}-1.$$
So Lemma \ref{lem 2} implies that
\begin{align}
\nu_2(S(c2^m,b2^{n+2}+a-i))\geq s_2(b2^{n+2}+a-i).\label{eq17}
\end{align}
Thus by Lemma \ref{lem -2}, (\ref{201}), (\ref{eq14}), (\ref{add5})
and (\ref{eq17}) one deduces that
\begin{align}
\nu_2(h(i,j))&\geq s_2(i)+s_2(j-i)+j-i+s_2(b2^{n+2}+a-j)-1\nonumber\\
&\geq s_2(j)+s_2(b2^{n+2}+a-j)+j-i-1\nonumber\\
&\geq s_2(b2^{n+2}+a)-1 \nonumber\\
&\geq s_2(a). \label{111}
\end{align}

If $1\leq a\leq 2^n$ and $a\leq i\leq j$, then $b2^{n+2}+a-i\le
b2^{n+2}\le 2^m$. Then by Lemma \ref{lem 1} one gets $\nu_2(S(c2^m,
b2^{n+2}+a-i))=s_2(b2^{n+2}+a-i)-1$. Hence by (\ref{eq14}),
(\ref{201}) and Lemma \ref{lem -2}, one has
\begin{align}
\nu_2(h(i,j))& =s_2(i)+s_2(j-i)+j-i+s_2(b2^{n+2}+a-j)-2\nonumber\\
&\geq s_2(j)+s_2(b2^{n+2}+a-j)+j-i-2\nonumber\\
&\geq s_2(b2^{n+2}+a)-2\nonumber\\
&\geq s_2(a)-1,\label{112}
\end{align}
with equality holding if and only if
\begin{align}\label{add4}
j=i, ~s_2(b)=1~{\rm and}~
s_2(b2^{n+2}+a-j)+s_2(j)=s_2(b2^{n+2}+a).
\end{align}
Since $1\leq j\leq 2^n$ and $a\leq i\leq j$, by Lemma \ref{lem 4} one
knows that (\ref{add4}) holds only when $i=j=a$ and $s_2(b)=1$. It
follows from (\ref{111}) and (\ref{112}) that $\nu_2(h(i, j))\ge
s_2(a)$ except for $i=j=a\in [1, 2^n]$ and $s_2(b)=1$, in which case
one has $\nu_2(h(a, a))=s_2(a)-1$. Then by (\ref{202}), one has
$\nu_2(\Delta _1)=s_2(a)-1$  if $a\in [1, 2^n]$ and $s_2(b)=1$, and
$\nu_2(\Delta _1)\ge s_2(a)$ otherwise. Thus Lemma \ref{lem 6} (i)
is true if $l=1$ and $m=n+2+\lfloor\log_2b\rfloor$. So the statement
for $\Delta _1$ is proved.

Now one handles $\Delta_4$. Note that $b2^{n+1}+1\leq j\leq
b2^{n+1}+2^n$ and $0\leq i\leq j$. Let $j=b2^{n+1}+j_0$ for some
integer $1\leq j_0\leq 2^n$. By Theorem \ref{thm 2} one has
\begin{align}
\nu_2\big(S(b2^{n+1}+2^n,j)\big)=\nu_2\big(S(b2^{n+1}+2^n,b2^{n+1}+j_0)\big)=
s_2(j_0)-1.\label{add6}
\end{align}

Since $m\geq n+2+\lfloor\log_2b\rfloor$, one has
$b2^{n+2}+a-j<b^{n+1}+a<2^m$. So by Lemmas \ref{lem 1} and \ref{lem
2} one gets
\begin{align} \nu_2\big(S(c2^m,b2^{n+2}+a-i)\big)\ge
s_2(b2^{n+2}+a-i)-1\label{add7}
\end{align}
and
\begin{align}\label{add8}
\nu_2\big(S(c2^m, b2^{n+2}+a-j)\big)=s_2(b2^{n+2}+a-j)-1
\end{align}
So by (\ref{201}), (\ref{eq14}), (\ref{add6})-(\ref{add8}) and
Lemma \ref{lem -2} one obtains that
\begin{align}\label{203}
\nu_2(h(i,j))&\geq s_2(i)+s_2(j-i)-s_2(j)+j-i+s_2(b2^{n+2}+a-j)+s_2(j_0)-2\nonumber\\
&\geq s_2(b2^{n+1}+a-j_0)+s_2(j_0)-2\nonumber\\
&\geq s_2(a)-1,
\end{align}
where equality holds if and only if $j=i$, $s_2(b)=1$ and
$s_2(b2^{n+1}+a-j_0)+s_2(j_0)=s_2(b2^{n+1}+a)$. It is similar to
$\Delta_1$ with $m\geq n+2+\lfloor\log_2b\rfloor$, by Lemma \ref{lem 5}
one has $\nu_2(\Delta _4)=s_2(a)-1$ if $a\in [1, 2^n]$ and
$s_2(b)=1$, and $\nu_2(\Delta _4)\ge s_2(a)$ otherwise. So Lemma
\ref{lem 6} (i) is true if $l=4$.

(ii). For $\Delta_2$, noticing that $2^n<j<2^{n+1}-1$, $0\leq i\leq j$ and
$m\geq n+2+\lfloor\log_2b\rfloor$, then by Lemmas \ref{lem -2}-\ref{lem 2}, one gets
\begin{align*}
\nu_2(S(c2^m,b2^{n+2}+a-i)S(b2^{n+1}+2^{n},j))\geq
s_2(b2^{n+2}+a-i)-1+s_2(j).
\end{align*}
So by (\ref{201}) and (\ref{eq14}), one has
\begin{align}
\nu_2(h(i,j)&\geq s_2(i)+s_2(j-i)+j-i+s_2(b2^{n+2}+a-j)-1\nonumber\\
&\geq s_2(b2^{n+2}+a)-1\nonumber\\
&\geq s_2(a).\label{204}
\end{align}
Hence by (\ref{202}) and (\ref{204}), one has $\nu_2(\Delta_2) \geq
s_2(a)$ as desired.

(iii). For $\Delta_3$, noting that $j=2^{n+1}-1$ and $0\leq
i\leq2^{n+1}-1$, it follows from Lemmas \ref{lem -2}-\ref{lem 2},
(\ref{202}) and (\ref{eq14}) that
\begin{align}
\nu_2(h(i,2^{n+1}-1)&\geq s_2(i)+s_2(j-i)+j-i+s_2(b2^{n+2}+a-j)-1\nonumber\\
&\geq s_2(b2^{n+2}+a-2^{n+1}+1)+s_2(2^{n+1}-1)-2\nonumber\\
&=s_2(b2^{n+2}+a-2^{n+1}+1)+n-1\nonumber\\
&\geq n,\label{205}
\end{align}
with equality holding if and only if $j=i=a=2^{n+1}-1$ and
$s_2(b)=1$. Since $1\le a<2^{n+1}$, one has $n+1\geq s_2(a)$. So by
(\ref{202}) and (\ref{205}), Lemma \ref{lem 6} (iii) follows
immediately.

This completes the proof of Lemma \ref{lem 6}.\hfill$\Box$\\  
\end{proof}

One can now use the lemmas presented in Section 2, Theorem \ref{thm 1} and
Lemma \ref{lem 6} to show Theorem \ref{thm 4}. The proof is of induction.\\

\noindent{\it Proof of Theorem 1.4.} By Lemma \ref{lem 0}, one gets that
\begin{align}
S(c2^m+b2^{n+1}+2^n,b2^{n+2}+a)&=\sum^{b2^{n+1}+2^n}_{j=1}\sum_{i=0}^{j}h(i,j)
=&\Delta_1+\Delta_2+\Delta_3+\Delta_4+\Delta, \label{000}
\end{align}
where $h(i,j)$ and $\Delta_l$ ($l=1,2,3,4$) are defined in
(\ref{201}) and (\ref{202}), respectively, and
\begin{align}
\Delta:=\sum_{j=2^{n+1}}^{b2^{n+1}}\sum_{i=0}^{j}h(i,j).\label{206}
\end{align}
First one deals with the 2-adic valuation of $h(i,j)$ with
$2^{n+1}\leq j\leq b2^{n+1}$ and $0\leq i\leq j$. Let
$j=j_12^{n+1}+j_2$ for some integers $1\leq j_1\leq b$ and $0\leq
j_2<2^{n+1}$.

If $j_2=0$, then $j=j_12^{n+1}$. So by Lemmas \ref{lem -2}-\ref{lem 2}
and (\ref{201}), one has
\begin{align}
\nu_2(h(i,j) &\geq
\nu_2\bigg(\frac{(b2^{n+2}+a-i)!}{(b2^{n+2}+a-j)!}\bigg)
+\nu_2(S(c2^m,b2^{n+2}+a-i))\nonumber\\
&\geq j-i+s_2(b2^{n+2}+a-j)-s_2(b2^{n+2}+a-i)+s_2(b2^{n+2}+a-i)-1\nonumber\\
&\geq s_2(b2^{n+2}+a-j)-1\nonumber\\
&=s_2((2b-j_1)2^{n+1}+a)-1\nonumber\\
&\geq s_2(a)\label{eq23}
\end{align}
since  $s_2(2b-j_1)\geq 1$ and $a<2^{n+1}$.

If $0<j_2<2^{n+1}$, by Theorem \ref{thm 1} one has
\begin{align}
\nu_2\big(S(b2^{n+1}+2^n,j)\big)=\nu_2\big(S(b2^{n+1}+2^n,j_12^{n+1}+j_2)\big)\geq
s_2(j_2)-1.\label{113}
\end{align}
Thus by Lemmas \ref{lem -2}-\ref{lem 1}, (\ref{201}), (\ref{eq14})
and (\ref{113}) one deduces
\begin{align}
\nu_2(h(i,j) &\geq
\nu_2\bigg(\frac{(b2^{n+2}+a-i)!}{(b2^{n+2}+a-j)!}\bigg)
+\nu_2(S(c2^m,b2^{n+2}+a-i))+\nu_2\big(S(b2^{n+1}+2^n,j)\big)\nonumber\\
& \geq j-i+s_2(b2^{n+2}+a-j)+s_2(j_2)-2\nonumber\\
&\geq s_2(j_2)+s_2((2b-j_1)2^{n+1}+a-j_2)-2\nonumber\\
&\geq s_2((2b-j_1)2^{n+1}+a)-2\nonumber\\
&=s_2(2b-j_1)+s_2(a)-2.\label{eq24}
\end{align}

Let $A_t$ be defined as in (\ref{add10}). Then
$\mathbb{Z}^+=\bigcup_{t=1}^\infty A_t$. One proves Theorem \ref{thm
3} by induction on $t$. First one considers that the case $t=1$. Let
$b\in A_1$. Then $s_2(b)=1$. If $0<j_2<2^{n+1}$, then $1\le j_1<b$.
So $s_2(2b-j_1)\ge 2$. Thus by (\ref{eq24}) one has that
$\nu_2(h(i,j))\geq s_2(a)$ if $0<j_2<2^{n+1}$. Furthermore, by
(\ref{206}) and (\ref{eq23}) one gets
\begin{align}
\nu_2(\Delta) \geq s_2(a).\label{006}
\end{align}
By Lemma \ref{lem 6} (iv), (\ref{000})  and (\ref{006}),
Theorem \ref{thm 3} for the case $s_2(b)=1$ follows immediately.
That is, Theorem \ref{thm 3} is proved if $t=1$.

Now let $t\geq 2$. Assume that Theorem \ref{thm 3} is true for any
integers $b\in A_1\cup\cdots\cup A_{t-1}.$  In what follows one
proves that Theorem \ref{thm 3} is true for the case $t$,
namely, for the case that $b\in A_t$.

For $b\in A_t$, let $b=2^{r_1}+2^{r_2}+...+2^{r_t}$
be the 2-adic expansion of $b$, where $r_1>r_2>\cdots>r_t$. Claim that
if $1\leq j_1<b$, then $s_2(2b-j_1)=1$ if and only if
$b=2^{r_1}+\frac{j_1}{2}$. One first notices that if $1\leq j_1<b$, then
$$2^{r_1+2}>2b>2b-j_1>2^{r_1}.$$
So $s_2(2b-j_1)=1$ if and only if  $2b-j_1=2^{r_1+1}$, i.e.,
$b=2^{r_1}+\frac{j_1}{2}$. The claim is proved. In the following one
handles $\Delta$. For this purpose, one needs to treat with $h(i, j)$.
Consider the following cases.

If $0<j_2<2^{n+1}$ and $s_2(2b-j_1)\geq 2$, then by (\ref{eq24}) one derives that
\begin{align}
\nu_2(h(i,j))\geq s_2(a).\label{209}
\end{align}

If $0<j_2<2^{n+1}$ and $s_2(2b-j_1)=1$, then by the claim one has
$b=2^{r_1}+\frac{j_1}{2}$. It then follows that
\begin{align}
S(b2^{n+1}+2^n,j_12^{n+1}+j_2)=S\big(2^{r_1+n+1}+\frac{j_1}{2}2^{n+1}+2^n,
\frac{j_1}{2}2^{n+2}+j_2\big). \label{114}
\end{align}
Since $2b-j_1=2^{r_1+1}$, one has $j_1=2^{r_2+1}+\cdots+2^{r_t+1}$,
which implies that $s_2(\frac{j_1}{2})=t-1$ and so $\frac{j_1}{2}\in A_{t-1}$.
Hence the inductive hypothesis applied to (\ref{114}) gives us that
\begin{align}
  \nu_2(S(b2^{n+1}+2^n,j_12^{n+1}+j_2))
  \left\{
  \begin{array}{ll}
=s_2(j_2)-1=n,&{\rm if} ~j_2=2^{n+1}-1,\\
\geq s_2(j_2),& {\rm if}~ 0<j_2<2^{n+1}-1.
  \end{array}
  \right.\label{115}
\end{align}
For $0<j_2<2^{n+1}-1$, it follows from Lemmas \ref{lem -2}-\ref{lem 2},
(\ref{201}), (\ref{eq14}) and (\ref{115}) that
\begin{align}
\nu_2(h(i,j))& \geq j-i+s_2(b2^{n+2}+a-j)+s_2(j_2)-1\nonumber\\
&\geq s_2(j_2)+s_2(b2^{n+2}-j_12^{n+1}+a-j_2)-1\nonumber\\
&\geq s_2((2b-j_1)2^{n+1}+a)-1\nonumber\\
&=s_2(a).\label{116}
\end{align}
For $j_2=2^{n+1}-1$, since $m\geq n+2+\lfloor\log_2b\rfloor=n+2+r_1$, one has
\begin{align}
b2^{n+2}+a-j=(2b-j_1)2^{n+1}+a-j_2=2^{n+2+r_1}+a-j_2\leq2^{n+r_1+2}\leq
2^m\label{207}.
\end{align}
Then by Lemma \ref{lem 1} and (\ref{207}) one deduces that
\begin{align}
S(c2^m,b2^{n+2}+a-j)=s_2(b2^{n+2}+a-j)-1.\label{208}
\end{align}
It then follows from $1\leq a< 2^{n+1}$, Lemmas \ref{lem -2}-\ref{lem 2}, (\ref{201}),
(\ref{eq14}), (\ref{115}) and (\ref{208}) that
\begin{align}
\nu_2(h(i,j)&\geq s_2(b2^{n+2}+a-j)-1+s_2(j_2)-1+ j-i\nonumber\\
&\geq s_2((2b-j_1)2^{n+1}+a-j_2)+ s_2(j_2)-2\nonumber\\
&=s_2((2b-j_1)2^{n+1}+a-2^{n+1}+1)+n-1\nonumber\\
&\geq n, \label{117}
\end{align}
with equality holding if and only if $j=i$, $s_2(2b-j_1)=1$ and $a=2^{n+1}-1$.

Finally, by (\ref{eq23}),(\ref{209}), (\ref{116}) and (\ref{117})
one obtains that if $b\in A_t$, then
\begin{align}
\nu_2(\Delta)\left\{
  \begin{array}{lc}
   =n&{\rm if} ~a=2^{n+1}-1,\\
\geq s_2(a)& {\rm if}~ a<2^{n+1}-1.
  \end{array}
  \right. \label{210}
\end{align}
Hence Lemma \ref{lem 6} (iv) together with (\ref{000}) and (\ref{210})
concludes that Theorem \ref{thm 3} is true if $b\in A_t$.

The proof of Theorem \ref{thm 3} is complete. \hfill$\Box$

\section{Proof of Theorem 1.3}

For any positive integer $k$, one defines $\theta(k)$ to be the
largest integer $l$ with $1\le l\le s_2(k)$ such that $\{m_{l},
m_{l-1}, ..., m_1\}$ is a set of consecutive integers, where
$k=2^{m_1}+2^{m_2}+\cdots+2^{m_{s_2(k)}}$ is the 2-adic expansion of
$k$ and $m_1>m_2>\cdots>m_{s_2(k)}$. Then
$\lceil\log_2k\rceil=m_1+1$. First Theorems \ref{thm 1} and
\ref{thm 3} are used to show the following lemma.

\begin{lem}\label{lem 8}
Let $n,k,a,c\in\mathbb{Z}^+$ be such that $3\leq k\le 2^n$, $s_2(k)\ge 2$ and $1\le
a\le\lceil\frac{k}{2}\rceil-1$. Suppose that $k$ is neither a power
of 2 nor a power of 2 minus 1. Then one has
$$\nu_2(S(c2^n-a,k-2a))=s_2(k)-\lceil\log_2k\rceil+\nu_2(a)$$
if either $a=\sum_{i=m_{\theta(k)}}^{m_1}2^{i-1}$ with $\theta(k)<s_2(k)$
or $a=\sum_{i=m_{\theta(k)}+1}^{m_1}2^{i-1}$ with $\theta(k)=s_2(k)$, and
$$\nu_2(S(c2^n-a,k-2a))>s_2(k)-\lceil\log_2k\rceil+\nu_2(a)$$ otherwise.
\end{lem}

\begin{proof}
First, one writes
\begin{align}\label{212}
k=\sum_{i=m_{\theta(k)}}^{m_1}2^i+\sum_{j=\theta(k)+1}^{s_2(k)}2^{m_j}.
\end{align}
Note that the second sum in (\ref{212}) vanishes if $\theta(k)=s_2(k)$.
Obviously, $m_1=m_l+l-1$ if $1\leq l\leq \theta(k)$ and $m_{\theta(k)}\ge
m_{{\theta(k)}+1}+2$ if $\theta(k)<s_2(k)$.

If $a=\sum_{i=m_{\theta(k)}}^{m_1}2^{i-1}$ with $\theta(k)<s_2(k)$,
then by (\ref{212}) one infers that
$k-2a=\sum_{j=\theta(k)+1}^{s_2(k)}2^{m_j}$ and
$\nu_2(c2^n-a)=\nu_2(a)=m_{\theta(k)}-1=m_1-\theta(k)$.
It then follows from $\lceil\log_2k\rceil=m_1+1$ that
\begin{align}\label{add16}
s_2(k-2a)=s_2(k)-\theta(k)
\end{align} and
\begin{align}\label{add17}\theta(k)=m_1-\nu_2(a)=\lceil\log_2k\rceil-1-\nu_2(a).
\end{align}
Since $m_{\theta(k)}\ge m_{{\theta(k)}+1}+2$, one has $k-2a<
2^{m_{\theta(k)}-1}=2^{\nu_2(c2^n-a)}$. It follows from Lemma
\ref{lem 1}, (\ref{add16}) and (\ref{add17}) that
\begin{align*}
\nu_2(S(c2^n-a,k-2a))
=s_2(k-2a)-1=s_2(k)-\lceil\log_2k\rceil+\nu_2(a)
\end{align*}
as required. Hence Lemma \ref{lem 8} is proved if
$a=\sum_{i=m_{\theta(k)}}^{m_1}2^{i-1}$ with $\theta(k)<s_2(k)$.

If $a=\sum_{i=m_{\theta(k)}+1}^{m_1}2^{i-1}$ with
$\theta(k)=s_2(k)$, then by (\ref{212}) one deduces that
$k-2a=2^{m_{\theta(k)}}$ and
$\nu_2(c2^n-a)=\nu_2(a)=m_{\theta(k)}=m_1+1-\theta(k)=\lceil\log_2k
\rceil-s_2(k)$ since $\lceil\log_2k\rceil=m_1+1$. Hence
$s_2(k)-\lceil\log_2k\rceil+\nu_2(a)=0$. It then follows from Lemma
\ref{lem 1} that
\begin{align*}
\nu_2(S(c2^n-a,k-2a))
=s_2(2^{m_{\theta(k)}})-1=0=s_2(k)-\lceil\log_2k\rceil+\nu_2(a).
\end{align*}
Thus Lemma \ref{lem 8} is proved if
$a=\sum_{i=m_{\theta(k)}+1}^{m_1}2^{i-1}$ with $\theta(k)=s_2(k)$.

Now one treats the remaining case that neither $a=\sum_{i=m_{\theta(k)}}
^{m_1}2^{i-1}$ with $\theta(k)<s_2(k)$ nor $a=\sum_{i=m_{\theta(k)}+1}
^{m_1}2^{i-1}$ with $\theta(k)=s_2(k)$. For this remaining case, one claims that
\begin{align}\label{add9}
\nu_2(S(c2^n-a,k-2a))\geq s_2(k)-m_1+\nu_2(a).
\end{align}
From the claim (\ref{add9}) and noting that
$\lceil\log_2k\rceil=m_1+1$, one derives that
\begin{align*}
\nu_2(S(c2^n-a,k-2a))>s_2(k)-\lceil\log_2k\rceil+\nu_2(a).
\end{align*}
So Lemma \ref{lem 8} holds for the remaining case that neither
$a=\sum_{i=m_{\theta(k)}}^{m_1}2^{i-1}$ with $\theta(k)<s_2(k)$
nor $a=\sum_{i=m_{\theta(k)}+1}^{m_1}2^{i-1}$ with $\theta(k)=s_2(k)$.
Thus one needs only to prove that the claim (\ref{add9}) is true,
which will be done in what follows.

If $\nu_2(a)<m_{s_2(k)}$, then $s_2(k)-(m_1-\nu_2(a))\leq
s_2(k)-(m_1-m_{s_2(k)}+1)\leq 0$ since $s_2(k)\le m_1-m_{s_2(k)}+1$.
This concludes that the claim (\ref{add9}) is true if
$\nu_2(a)<m_{s_2(k)}$.

If $m_{s_2(k)}\leq \nu_2(a)<m_{\theta(k)}-1$, then
$\theta(k)<s_2(k)$ and there is exactly one integer $t$ with
$\theta(k)<t\le s_2(k)$ such that $m_{t}\leq \nu_2(a)<m_{t-1}$. Then
by the definition of $\theta(k)$ one knows that $\{\nu_2(a),
m_{t-1},..., m_{\theta(k)}, ..., m_1\}$ is not consisting of
consecutive integers. This implies that
$s_2(2^{m_1}+\cdots+2^{m_{t-1}}+2^{m_t})=s_2(2^{m_1}+\cdots
+2^{m_{t-1}}+2^{\nu_2(a)})\le m_1-\nu_2(a)$. Therefore
\begin{align}
s_2(2^{m_t}+\cdots+2^{m_{s_2(k)}})&=
s_2(k)-s_2(2^{m_1}+\cdots+2^{m_{t-1}}+2^{m_t})+1\nonumber\\
&\geq s_2(k)-(m_1-\nu_2(a))+1.\label{add15}
\end{align}
Since $\nu_2(c2^n-a)=\nu_2(a)$ and $m_{t}\leq \nu_2(a)<m_{t-1}$, one
may write $c2^n-a=c_12^{\nu_2(a)}$ and
$k-2a=c_22^{\nu_2(a)+1}+2^{m_t}+\cdots+2^{m_{s_2(k)}}$ with $c_1$
and $c_2$ being integers. Then by Theorem \ref{thm 1} and
(\ref{add15}) one deduces that
\begin{align*}
\nu_2(S(c2^n-a,k-2a))&=\nu_2(S(c_12^{\nu_2(a)},
c_22^{\nu_2(a)+1}+2^{m_t}+\cdots+2^{m_{s_2(k)}})\\
&\geq s_2(2^{m_t}+\cdots+2^{m_{s_2(k)}})-1 \\
&\geq s_2(k)-m_1+\nu_2(a)
\end{align*}
as desired. Hence the claim (\ref{add9}) is proved if
$m_{s_2(k)}\leq \nu_2(a)<m_{\theta(k)}-1$.

If $m_{\theta(k)}-1\leq \nu_2(a)\leq m_1-1$, then by (\ref{212}) one can write
\begin{align}
k-2a&=\sum_{i=\nu_2(a)+1}^{m_1}2^i-2a+u=b2^{\nu_2(a)+2}+u\label{eq32}
\end{align}
and
\begin{align}
c2^n-a&=c_32^{m_1}+\sum_{i=\nu_2(a)}^{m_1-1}2^i+2^{\nu_2(a)}-a
=c_32^{m_1}+b2^{\nu_2(a)+1}+2^{\nu_2(a)},\label{eq33}
\end{align}
where $c_3\in\mathbb{Z}^+$ and $u$ and $b$ are defined as follows:
\begin{align}\label{ad3}
u:=\sum_{i=m_{\theta(k)}}^{\nu_2(a)}2^i+\sum_{j=\theta(k)+1}^{s_2(k)}2^{m_j},
~~~~~~
b:=\bigg(\sum_{i=\nu_2(a)}^{m_1-1}2^i-a\bigg)\bigg/2^{\nu_2(a)+1}.
\end{align}
Note that the first sum of $u$ vanishes if
$\nu_2(a)=m_{\theta(k)}-1$ and the second sum of $u$ vanishes if
$\theta(k)=s_2(k)$. By (\ref{212}), one has
\begin{align}\label{213}
s_2(u)=s_2(k)-s_2\big(\sum_{i=\nu_2(a)+1}^{m_1}2^i\big)=s_2(k)-m_1+\nu_2(a).
\end{align}
In the following one shows that $u<2^{\nu_2(a)+1}-1$. If
$\theta(k)=s_2(k)$, then by (\ref{212}) one has
$k=\sum_{i=m_{\theta(k)}}^{m_1}2^i$. But $k$ is not a power of 2
minus 1. So $m_{\theta(k)}\geq1$. Thus by (\ref{ad3}) one knows that
$u=\sum_{i=m_{\theta(k)}}^{\nu_2(a)}2^i<2^{\nu_2(a)+1}-1$. If
$\theta(k)<s_2(k)$, then $m_{\theta(k)}\ge m_{{\theta(k)}+1}+2$.
Hence by (\ref{ad3}) one yields that $u<2^{\nu_2(a)+1}-1$. Suppose that
$b<0$. Then from (\ref{eq32}) one deduces that
$$
k-2a\le -2^{\nu_2(a)+2}+u<-2^{\nu_2(a)+2}+2^{\nu_2(a)+1}-1<0,
$$
which is impossible. So $b\geq 0$.

If $b>0$, then by (\ref{ad3}) one has $m_1\ge
\nu_2(a)+2+\lfloor\log_2 b\rfloor.$ Since $u<2^{\nu_2(a)+1}-1$, it
then follows from (\ref{eq32})-(\ref{213}) and Theorem \ref{thm 3} that
\begin{align*}
\nu_2(S(c2^n-a,k-2a))&=\nu_2\big(S\big(c_32^{m_1}+b2^{\nu_2(a)+1}+
2^{\nu_2(a)},b{2^{\nu_2(a)+2}}+u\big)\big)\\
&\geq s_2(u)=s_2(k)-m_1+\nu_2(a).
\end{align*}
The claim (\ref{add9}) is proved if $m_{\theta(k)}-1\leq
\nu_2(a)\leq m_1-1$ with $b>0$.

If $b=0$, then by (\ref{eq32}) one has $u>0$ since $k-2a>0$. In what
follows one shows that $u>2^{\nu_2(a)}$. Suppose that $0<u\le
2^{\nu_2(a)}$. From (\ref{ad3}) one infers that either
$\theta(k)<s_2(k)$ with $\nu_2(a)=m_{\theta(k)}-1$, or
$\theta(k)=s_2(k)$ with $\nu_2(a)=m_{\theta(k)}$. If
$\theta(k)<s_2(k)$ with $\nu_2(a)=m_{\theta(k)}-1$, then by
(\ref{ad3}) one gets $a=\sum_{i=m_{\theta(k)}-1}^{m_1-1}2^i$ since
$b=0$. It contradicts with the assumption that
$a\neq\sum_{i=m_{\theta(k)}}^{m_1}2^{i-1}$ if $\theta(k)<s_2(k)$. If
$\theta(k)=s_2(k)$ with $\nu_2(a)=m_{\theta(k)}$, it then follows
from (\ref{ad3}) and $b=0$ that
$a=\sum_{i=m_{\theta(k)}}^{m_1-1}2^i$, which contradicts with the
assumption that $a\neq\sum_{i=m_{\theta(k)}+1}^{m_1}2^{i-1}$ if
$\theta(k)<s_2(k)$. Hence $u>2^{\nu_2(a)}$. Note that
$u<2^{\nu_2(a)+1}-1$. Now by (\ref{eq32})-(\ref{213}) and Lemma
\ref{lem 2} one deduces that
\begin{align*}
\nu_2(S(c2^n-a,k-2a))=\nu_2\big(S\big(c_32^{m_1}+
2^{\nu_2(a)},u\big)\big) \geq s_2(u)=s_2(k)-m_1+\nu_2(a)
\end{align*}
as desired. The claim (\ref{add9}) is proved if $m_{\theta(k)}-1\leq
\nu_2(a)\leq m_1-1$ with $b=0$.

This concludes the proof of Lemma \ref{lem 8}. \hfill$\Box$
\end{proof}\\

One is now in a position to show Theorem 1.3.\\

\noindent{\it Proof of Theorem 1.3.} Suppose that (\ref{c2}) is true. Then
using (\ref{c2}) with $a=2c$ and $b=c$, one can easily derive that (\ref{c1})
holds. So one only needs to show that (\ref{c2}) is true, which will be done
in the following.

To prove (\ref{c2}), one uses (\ref{1}) and (\ref{2}) with $p=2$, $m=(2b-a)2^n$,
$v=1$ and $n$ replaced by $(a-b)2^n$, and considers the coefficients of $x^k$:
\begin{align}\label{add13}
S(a2^n,k)&\equiv \sum_{j=0}^{(a-b)2^n}{(a-b)2^n\choose j}
S(j+(2b-a)2^n,k-2((a-b)2^n-j))\nonumber\\
&=S(b2^n,k)+\sum_{j=(a-b)2^n-\lceil\frac{k}{2}\rceil+1}^{(a-b)2^n-1}{(a-b)2^n\choose j}
S(j+(2b-a)2^n,k-2((a-b)2^n-j))\nonumber\\
&=S(b2^n,k)+\sum_{i=1}^{\lceil\frac{k}{2}\rceil-1}{(a-b)2^n\choose
i}S(b2^n-i,k-2i)\mod2^{n+\nu_2(a-b)}.
\end{align}
In then follows from (\ref{add13}) that
\begin{align}\label{add14}
S(a2^n,k)-S(b2^n,k)&\equiv\sum_{i=1}^{\lceil\frac{k}
{2}\rceil-1}{(a-b)2^n\choose i}S(b2^n-i,k-2i)\mod2^{n+\nu_2(a-b)}.
\end{align}

In what follows one discusses the 2-adic valuation of a general term of
(\ref{add14}) with $1\le i\le \lceil\frac{k}{2}\rceil-1 $. Let
$a-b=c_02^{\nu_2(a-b)}$ with $c_0\geq 1$ being odd. One first notices
that $i\leq\lceil\frac{k}{2}\rceil-1 <2^n$. So by Lemma \ref{lem 7}
one infers that
\begin{align}
s_2(c_02^{n+\nu_2(a-b)}-i)=s_2(c_0)+n+\nu_2(a-b)-\nu_2(i)-s_2(i).\label{ad4}
\end{align}
It then follows from Lemma \ref{lem -2} and (\ref{ad4}) that
\begin{align}
&\nu_2\big({(a-b)2^n\choose i}S(b2^n-i,k-2i)\big)\nonumber\\
&=s_2(i)+s_2(c_02^{n+\nu_2(a-b)}-i)-s_2(c_02^{n+\nu_2(a-b)})
+\nu_2(S(b2^n-i, k-2i))\nonumber\\
&=n+\nu_2(a-b)-\nu_2(i)+\nu_2(S(b2^n-i,k-2i)).\label{add18}
\end{align}
One considers the following two cases.

{\it Case 1.} $s_2(k)=1$. Then one may write $k=2^m$. If $m=2$, then
by (\ref{1.1}) one has
\begin{align*}
&\nu_2(S(a2^n,4)-S(b2^n,4))\nonumber\\
=&\nu_2\Big(\frac{1}{6}(4^{a2^n-1}-3^{a2^n}+3\cdot2^{a2^n-1}-1)
-\frac{1}{6}(4^{b2^n-1}-3^{b2^n}+3\cdot2^{b2^n-1}-1)\Big)\nonumber\\
=&\nu_2(3^{b2^n}(3^{(a-b)2^n}-1))-1
\end{align*}
By Lemma \ref{lem 9}, one has $\nu_2(3^{(a-b)2^n}-1)=n+\nu_2(a-b)+2$. It follows that
$$\nu_2(S(a2^n,4)-S(b2^n,4))=n+\nu_2(a-b)-\lceil\log_24\rceil+s_2(4)+\delta(4)$$
since $\delta(4)=2$. Namely, Theorem \ref{thm 4} holds if $m=2$.

Now let $m\ge 3$. So $ 1\le i\le 2^{m-1}-1$. If $i=2^{m-2}$, then by Lemma \ref{lem -3}
\begin{align}\label{add19}
\nu_2(S(b2^n-i,2^m-2i))=\nu_2(S(b2^n-2^{m-2},2^{m-1}))=0.
\end{align}
Thus by (\ref{add18}) and (\ref{add19}) one obtains that
\begin{align}
\nu_2\big({(a-b)2^n\choose i}S(b2^n-i,2^m-2i)\big)
=n+\nu_2(a-b)-(m-2).\label{add20}
\end{align}

If $i\neq 2^{m-2}$, then $\nu_2(i)<2^{m-2}$ since $i\le 2^{m-1}-1$.
It then follows from (\ref{add18}) that
\begin{align}
\nu_2\big({(a-b)2^n\choose i}S(b2^n-i,2^m-2i)\big)
&>n+\nu_2(a-b)-(m-2)+\nu_2(S(b2^n-i,2^m-2i))\nonumber\\
&\ge n+\nu_2(a-b)-(m-2).\label{add21}
\end{align}
Hence by (\ref{add14}), (\ref{add20}) and (\ref{add21}) one derives that
\begin{align*}
\nu_2(S(a2^n,2^m)-S(b2^n,2^m))&=n+\nu_2(a-b)-m+2\\
&=n+\nu_2(a-b)-\lceil\log_22^m\rceil+s_2(2^m)+\delta(2^m)
\end{align*}
since $\delta(2^m)=1$. So (\ref{c2}) is true if $s_2(k)=1$.

{\it Case 2.} $s_2(k)\ge 2$. Since $k$ is neither a power of 2 nor
a power of 2 minus 1 and $1\leq i\le\lceil\frac{k}{2}\rceil-1$, by
Lemma \ref{lem 8}, (\ref{add14}) and (\ref{add18}) one obtains that
\begin{align*}
\nu_2(S(a2^n,k)-S(b2^n,k))&=n+\nu_2(a-b)-\lceil\log_2k\rceil+s_2(k)\\
&=n+\nu_2(a-b)-\lceil\log_2k\rceil+s_2(k)+\delta(k).
\end{align*}
since $\delta(k)=0$. Hence (\ref{c2}) holds in this case.

The proof of Theorem \ref{thm 4} is complete. \hfill$\Box$\\

\noindent{\bf Remark 5.1.} By Theorem \ref{thm 4}, one knows that
Conjecture 1.1 is true if $k$ is not a power of 2 minus 1. Theorem
\ref{thm 4} tells us that
$\nu_2\Big(S(a2^{n+1},k)-S(b2^{n+1},k)\Big)<n+\nu_2(a-b)$ if $k\neq
2^m-1$ and $k\neq 4$. In fact, in the proof of Theorem \ref{thm
4}, to handle the case that $k\neq 2^m-1$ and $k\neq 4$, one makes
use of the Junod's congruence (\ref{1}). However, for the remaining
case $k=2^m-1$, numerical experimentation (see \cite{[Le2]}) suggests
that
$$\nu_2\Big(S(a2^{n+1},2^m-1)-S(b2^{n+1},2^m-1)\Big)=n+1+\nu_2(a-b)>n+\nu_2(a-b).$$
Thus, to get such result, the modulus in Lemma \ref{lem 10} (and so
(\ref{add14}) above) is not enough. Hence one has to find a
congruence stronger than (\ref{1}). Unfortunately, one encounters
difficulties in strengthening congruence (5). Maybe one needs some
new approaches to attack Conjecture 1.1 for the remaining case $k=2^m-1$.\\

\noindent{\bf Acknowledgements} The authors would like to thank the anonymous
referees for careful reading of the manuscript and for helpful comments and
suggestions which improved its presentation.

\end{document}